\documentclass[oneside,english,draft]{amsart}
\usepackage[T1]{fontenc}
\usepackage[latin9]{inputenc}
\usepackage{babel}
\usepackage{verbatim}
\usepackage{url}
\usepackage{amsthm}
\usepackage{amssymb}
\usepackage{booktabs}
\usepackage[ruled, vlined, linesnumberedhidden]{algorithm2e}
\usepackage[draft=false, unicode=true,
 bookmarks=false,bookmarksnumbered=false,bookmarksopen=false,
 breaklinks=false,pdfborder={0 0 1},backref=false,colorlinks=false]
 {hyperref}
\hypersetup{
 pdfauthor={Giosu{\`e} Muratore}}

\usepackage{lineno}
\usepackage{xcolor}

\makeatletter
\numberwithin{equation}{section}
\numberwithin{figure}{section}
\numberwithin{table}{section}
\theoremstyle{plain}
\newtheorem{thm}{\protect\theoremname}[section]
\theoremstyle{definition}
\newtheorem{example}[thm]{\protect\examplename}
\theoremstyle{remark}
\newtheorem{rem}[thm]{\protect\remarkname}
\theoremstyle{definition}
\newtheorem{defn}[thm]{\protect\definitionname}
\theoremstyle{plain}
\newtheorem{lem}[thm]{\protect\lemmaname}
\usepackage{tikz}
\usetikzlibrary{positioning}

\makeatother

\providecommand{\definitionname}{Definition}
\providecommand{\examplename}{Example}
\providecommand{\remarkname}{Remark}
\providecommand{\theoremname}{Theorem}
\providecommand{\lemmaname}{Lemma}

\begin{document}
\global\long\def\d{\mathrm{d}}%
\global\long\def\ev{\mathrm{ev}}%
\global\long\def\ctopT{c_{\mathrm{top}}^{T}}%
\global\long\def\ctop{c_{\mathrm{top}}}%
\global\long\def\ch{\mathrm{ch}}%
\global\long\def\td{\mathrm{td}}%

\title{Effective computations of the Atiyah-Bott formula}
\author{Giosu{\`e} Muratore and Csaba Schneider}
\date{\today}
\begin{abstract}
We present an implementation of the Atiyah-Bott residue formula for $\overline{M}_{0,m}(\mathbb{P}^{n},d)$.
We use this implementation to compute a large number of Gromov-Witten invariants of
genus $0$, including intersection numbers of rational curves on
general complete intersections. We also compute some
numbers of rational contact curves satisfying
suitable Schubert conditions. Our computations confirm known
predictions made by Mirror Symmetry. The code we developed for these problems is publicly available and can also be 
used for other types of computations.
\end{abstract}

\address{Department of Mathematics, Universidade Federal de Minas Gerais, Belo
Horizonte, MG, Brazil.}
\email{\href{mailto:muratore.g.e@gmail.com}{muratore.g.e@gmail.com}, \href{mailto:csaba.schneider@gmail.com}{csaba.schneider@gmail.com}}
\urladdr{\url{https://sites.google.com/view/giosue-muratore}}
\maketitle

\section{Introduction}

The solution of an enumerative problem in algebraic geometry is often given as the degree of a polynomial of Chern classes of bundles over some moduli space $M$.
The computation of such classes is easier when the moduli space and the
bundle carry an action of a group $G$. Moreover, the $G$-action 
on $M$ lifts equivariantly to the total space of many natural
bundles (such as the tangent bundle).

Suppose that the fixed locus of the action on $M$ is a finite union
of points $\{V\}_{V\in I}$, and $i:V\hookrightarrow M$ is the inclusion.
The Atiyah-Bott Theorem (see \cite{bott1967residue,AB}) says that the localized top Chern class of the normal
bundle $\ctop(N_{V/M})$ is invertible along the fixed locus of the
group action. Thus, for any equivariant class $P$ of $M$, we have
the celebrated formula:
$$
\int_{M}P=\sum_{V\in I}\int_{V}^{G}\frac{i^{*}(P)}{\ctop(N_{V/M})}.
$$
The symbol $\int_{V}^{G}$ denotes the $G$-equivariant integral at $V$. 
The terms in the sum are often easier to compute than the degree on the 
left-hand side. The first application of this formula for enumerative problems
is due to 
\cite{ES}. The results in~\cite{ES} include complete lists
of numbers of twisted cubic curves and elliptic quartic
curves contained in a general complete intersection satisfying
suitable Schubert conditions. Kontsevich~
\cite{kontsevich1995enumeration} used an analogous formula for the moduli of stable maps
$\overline{M}_{0,m}(\mathbb{P}^{n},d)$, for arbitrary $n$, $m$
and $d$. He successfully computed the number of rational quartics
contained in a general quintic threefold, confirming the prediction
of Mirror Symmetry. His version of the formula is
$$
\int_{\overline{M}_{0,m}(\mathbb{P}^{n},d)}P=\sum_{\Gamma}\frac{1}{a_{\Gamma}}
\frac{P^{T}(\Gamma)}{\ctopT(N_{\Gamma})(\Gamma)},
$$
where the right-hand side is a finite sum of rational expressions over a set of combinatorial data (see Section $3$). This approach was later generalized to positive
genus by \cite{graber1999localization}. Such results are well known
to experts in enumerative geometry.

One of the ingredients of Kontsevich's method is parameterizing
the fixed locus of the torus action on $\overline{M}_{0,m}(\mathbb{P}^{n},d)$
by graphs colored by the fixed points of the action on $\mathbb{P}^{n}$.
All such graphs can be obtained in a few seconds by modern
computers, even though their number grows very fast. Any intersection
of Chern classes of equivariant vector bundles will be a complicated
sum of rational expressions taken over the set of all such graphs. At the very end,
however, this sum magically collapses to a rational number, as expected.

The authors implemented computer code to apply the Atiyah-Bott formula to $\overline{M}_{0,m}(\mathbb{P}^{n},d)$.
The package can be used to to perform a wide variety of computations as long as the equivariant class is given with an explicit formula. 
Such equivariant classes can be easily implemented inside the package,  and so our code can perform analogous
computations not described here. Finally, we
found classical (and less classical) numbers of curves with geometric
conditions. In particular, we found a complete list of numbers of contact curves (that is, the integral curves of a variety with contact structure) meeting arbitrary linear subspaces. The enumeration of such curves, started by \cite{levVa}, was only partially known before. Our main result is thus the following theorem.
\begin{thm}
The numbers of rational contact curves in $\mathbb{P}^{3}$ up to degree $8$ meeting arbitrary linear subspaces are given in Table \ref{tab:contact}.
\end{thm}
This list includes contact curves meeting
an appropriate number of lines in $\mathbb{P}^{3}$, as described
in \cite[A328553]{oeis}.

In Section 2 we review the notation and recall some basic results. In Sections~3 and 4 we briefly describe the Atiyah-Bott formula and our algorithm. In the last section, we compute several Gromov-Witten invariants, including the virtual numbers of rational curves of the quintic threefold, up to degree $9$. Some of these numbers were never computed before using the Atiyah-Bott formula. Hiep \cite{Hiep} computed them up to degree $6$ (but it is not clear, for us, how his code can be used for computations that are different from those explicitly given in the article). In the same section there are many other applications of our code, including the number of curves with tangency conditions.
Our code is written in the 
Julia programming language \cite{bezanson2017julia} and is freely available.

\smallskip

\noindent {\bf Acknowledgments.} The first author thanks Rahul Pandharipande, Andrea Ricolfi, Xiaowen Hu and Jieao Song for their help, and Alfredo Donno and Matteo Cavaleri for a fruitful discussion.
He was also supported by a postdoctoral fellowship PNPD-CAPES.
The second author acknowledges the 
support of the CNPq projects {\em Produtividade em Pesquisa} (project no.: 308212/2019-3)  
and {\em Universal} (project no.: 421624/2018-3).
We both are grateful to Israel Vainsencher for his continuous support, and to the referee for their careful reading.

\section{Stable maps}

\subsection{Moduli space and invariants}
A marked stable map to $\mathbb P^n$ of degree $d$ is a tuple $(C,f,p_1,\ldots,p_m)$ in which the components are as follows.
\begin{itemize}
    \item $C$ is a projective, connected, reduced, at-worst-nodal complex curve of arithmetic genus zero.
    \item $f\colon C\rightarrow \mathbb P^n$ is a morphism such that $f_*[C]$ is a cycle of degree $d$.
    \item $p_1,\ldots,p_m$ are regular distinct points of $C$, called marked points.
    \item If $E\subset C$ is a component contracted by $f$, then $E$ contains at least three points among the marks and nodes of $C$.
\end{itemize}
The coarse moduli space of such stable maps is denoted by $\overline{M}_{0,m}(\mathbb{P}^{n},d)$. It is irreducible of dimension $n+(n+1)d+m-3$, see \cite{FP}.

In this section we will use the following notation.
\begin{itemize}
\item $\ev_{i}\colon \overline{M}_{0,m}(\mathbb{P}^{n},d)\rightarrow\mathbb{P}^{n}$
is the evaluation map at the $i^{\mathrm{th}}$ marked point.
\item $\delta\colon \overline{M}_{0,m+1}(\mathbb{P}^{n},d)\rightarrow\overline{M}_{0,m}(\mathbb{P}^{n},d)$
is the forgetful map of the last point.
\end{itemize}

Let $h$ be the cohomology class of a hyperplane of $\mathbb{P}^{n}$, and let $a_{1},\ldots,a_{m}$
be positive integers, such that $\sum_{i=1}^m a_{i}=\dim\overline{M}_{0,m}(\mathbb{P}^{n},d)$. The following Gromov-Witten invariant counts the number of curves in $\mathbb P^n$ of degree $d$ meeting $m$ linear subspaces of codimension $a_i$ in general position for $i=1,\ldots,m$ \cite[Lemma 14]{FP}:
\begin{equation}\label{eq:GW_std}
\int_{\overline{M}_{0,m}(\mathbb{P}^{n},d)}
\ev_1^*(h^{a_1})\cdot\cdot\cdot\ev_m^*(h^{a_m}).
\end{equation}
More generally, we may consider the descendant invariants given as
\begin{equation}\label{eq:PsiGW_std}
\int_{\overline{M}_{0,m}(\mathbb{P}^{n},d)}
\ev_1^*(h^{a_1})\cdots\ev_m^*(h^{a_m})\cdot\psi^{a'_1}_1\cdots\psi^{a'_m}_m,
\end{equation}
where $\psi_i$ is the first Chern class of the line bundle $\mathbb L_i$ whose fiber at a moduli point is the cotangent space of the curve at the $i^\mathrm{th}$ marked point.
\begin{rem}\label{rem:semplif}
Using the projection formula, it is not
difficult to see that (\ref{eq:GW_std}) is equal to
\begin{equation}\label{eq:GW_incomplete}
\int_{\overline{M}_{0,0}(\mathbb{P}^{n},d)}\delta_{*}(\ev_1^*(h^{a_1}))\cdots\delta_{*}(\ev_1^*(h^{a_m})).
\end{equation}
In the next subsection we will see that by reducing the number of marks we will speed up the computations for the Atiyah-Bott formula. Indeed, by  \cite[Theorem 4.3]{muratore2020enumeration}, $\delta_{*}(h^{a_{1}})$ is a polynomial in the Chern classes
of equivariant vector bundles;
see \cite[Proposition 3.4]{ES} for a similar discussion.
\end{rem}
\begin{rem}\label{rem:jet}
Enumerative problems of curves with tangency conditions often involve the bundles $\mathrm{J}^p(\ev_i^*\mathcal{O}_{\mathbb{P}^n}(z))$ of $p$-jets (see \cite[Section 7.2]{3264} for the definition). The Euler class of this bundle can be obtained by recursion using the condition $\mathrm{J}^0(\ev_i^*\mathcal{O}_{\mathbb{P}^n}(z)):=\ev_i^*\mathcal{O}_{\mathbb{P}^n}(z)$ and by the exact sequence

\begin{equation*}
0\rightarrow\ev_i^*\mathcal{O}_{\mathbb{P}^n}(z)\otimes\mathbb L^{\otimes p}_i \rightarrow
\mathrm{J}^{p}(\ev_i^*\mathcal{O}_{\mathbb{P}^n}(z))\rightarrow
\mathrm{J}^{p-1}(\ev_i^*\mathcal{O}_{\mathbb{P}^n}(z))\rightarrow 0.
\end{equation*}
\end{rem}

Let $X\overset{i}{\hookrightarrow}\mathbb{P}^{n}$ be a general complete
intersection of multidegree $b=(b_{1},\ldots,b_{s})$, such that $n+1\ge\sum b_{i}$.
That is, $X$ is the zero locus of a general section $\sigma$ of
the vector bundle $F=\bigoplus_{i=1}^{s}\mathcal{O}_{\mathbb{P}^{n}}(b_{i})$.
By standard base change theory, $\delta_{*}(\ev_{m+1}^{*}F)$ is a
vector bundle on $\overline{M}_{0,m}(\mathbb{P}^{n},d)$ of rank $\sum(b_{i}d+1)=s+\sum b_{i}d$.
The section defining $X$ induces in a natural way a section $\sigma'$
of $\delta_{*}(\ev_{m+1}^{*}F)$ whose zero locus is the locus in
$\overline{M}_{0,m}(\mathbb{P}^{n},d)$ of maps with image contained in $X$. In other words, the section $\sigma'$
vanishes at a moduli point of a stable map $f:C\rightarrow\mathbb{P}^{n}$
if and only if $f=i\circ f'$ for some map $f':C\rightarrow X$, and
$i_{*}f_{*}[C]$ has degree $d$.

We may consider the Gromov-Witten invariant
\begin{equation}\label{eq:GW_complete}
\int_{\overline{M}_{0,m}(\mathbb{P}^{n},d)}\ev_{1}^{*}(h^{a_{1}})\cdots\ev_{m}^{*}(h^{a_{m}})\cdot\ctop(\delta_{*}(\ev_{m+1}^{*}F)).
\end{equation}
This invariant counts the virtual number of rational curves\footnote{That is, the number of stable maps with multiplicity.} in $X$
meeting general hyperplanes of codimension $a_{i}$ for $i=1,\ldots,m$.%

\begin{rem}
Equation (\ref{eq:GW_complete}) does not coincide with the number of rational curves in $X$
meeting general hyperplanes of codimension $a_{i}$ for $i=1,\ldots,m$ if we have contribution,
for example, from stable maps with automorphisms. That is, (\ref{eq:GW_complete}) is not enumerative.
The Aspinwall--Morrison
formula says that if $X$ is a Calabi-Yau threefold and $d=k_1k_2$, then the
contribution to (\ref{eq:GW_complete}) from smooth rational curves
of degree $k_1$ is $k_2^{-3}$, see \cite{aspinwall1993topological,voisin1996mathematical}.
Manin \cite{manin95} computed such a contribution by proving that:
\begin{equation}\label{eq:manin}
    \int_{\overline{M}_{0,0}(\mathbb{P}^{1},d)} \ctop(R^1\delta_{*}\ev_1^*(\mathcal{O}_{\mathbb P^1}(-1))^{\oplus 2})=\frac{1}{d^3}. 
\end{equation}
\end{rem}

When $X$ is Fano and convex, then it is also homogeneous \cite{Pand}. In
that case, we can apply \cite[2 Theorem]{Kle} to make the
intersection in (\ref{eq:GW_complete}) be transverse, so (\ref{eq:GW_complete})
is enumerative.

Let $n$ be odd. It is well-known that $\mathbb{P}^{n}$ has a unique contact structure\footnote{A contact structure on a smooth complex variety $X$ of dimension $2n+1$ is a line subbundle $L \hookrightarrow\Omega_{X}$ such that for every local section $s$ of $L$, the section $s\wedge(\mathrm{d}s)^{\wedge n}$ is nowhere zero.} given by the line subbundle $\mathcal{O}_{\mathbb{P}^{n}}(-2) \hookrightarrow\Omega_{\mathbb{P}^{n}}$, see \cite{kob}. The curves pointwise tangent to the contact distribution are particularly interesting in the classification of contact structures (see, e.g., \cite{ye}). Such curves are called contact.

\begin{defn}
A rational curve $C\subset\mathbb{P}^{n}$
of degree $d$ is contact if $s_{|TC}\equiv0$ for every local section $s$ of $\mathcal{O}_{\mathbb{P}^{n}}(-2)$.
\end{defn}

Let $\omega_\delta$ be the relative dualizing sheaf of
$\delta\colon\overline{M}_{0,m+1}(\mathbb{P}^{n},d)\rightarrow\overline{M}_{0,m}(\mathbb{P}^{n},d)$. We may define the equivariant vector bundle over $\overline{M}_{0,m}(\mathbb{P}^{n},d)$

$$\mathcal{E}_{d,m}:=\delta_{*}(\omega_{\delta}\otimes\ev_{m+1}^{*}\mathcal{O}_{\mathbb{P}^{n}}(2)).$$

The first author proved in \cite{muratore2020enumeration} that the number of
rational contact curves of degree $d$ meeting general linear
hyperplanes of codimension $a_{i}$ for $i=1,\ldots,m$ is
\begin{equation}\label{eq:contact}
\int_{\overline{M}_{0,m}(\mathbb{P}^{n},d)}\ev_{1}^{*}(h^{a_{1}})\cdots\ev_{m}^{*}(h^{a_{m}})\cdot\ctop(\mathcal{E}_{d,m}).
\end{equation}
We will be interested in rational contact space curves of degree $d$ meeting $a$ general points and $b$ general lines, where $a+2b=2d+1$. 
\begin{rem}
Both (\ref{eq:GW_complete}) and (\ref{eq:contact}) can be reduced to the case when $m=0$, thanks to Remark \ref{rem:semplif}, which holds in these cases as well.
\end{rem}

\section{Localization}\label{sec:3}
In this section we will briefly recall the powerful method of localization, following \cite{kontsevich1995enumeration} and \cite{cox1999mirror}.
Let $T=(\mathbb{C}^{*})^{n+1}$ be the torus acting on $\mathbb{P}^{n}$ with the standard action, and let $\{x_{i}\}_{i=0}^{n}$ be the fixed points.
The action of the torus $T$ on $\mathbb{P}^{n}$ lifts to $\overline{M}_{0,m}(\mathbb{P}^{n},d)$
by composition. That is, if $f\colon C\rightarrow \mathbb{P}^{n}$ is a stable map, for every $t\in T$ we have a new stable map $t\cdot f\colon C\rightarrow \mathbb{P}^{n}$. The fixed locus of this action is usually denoted by $\overline{M}_{0,m}(\mathbb{P}^{n},d)^T$, its connected components are finite and naturally labelled by isomorphism classes of combinatorial data $\Gamma=(g, \mathbf{c}, \mathbf{w}, \mathbf{q})$, called decorated graphs defined as follows.
\begin{enumerate}
\item $g$ is a tree; that is $g$ is a simple undirected, connected graph without cycles. We denote by $V_\Gamma$ and $E_\Gamma$ the sets of vertices and edges of $g$.
\item The vertices of $g$ are colored by the points $\{x_{i}\}_{i=0}^{n}$; the coloring is given by a map $\mathbf c\colon V_\Gamma \rightarrow \{x_0,\ldots,x_n\}$ such that $\mathbf{c}(v)\neq\mathbf{c}(w)$ if $v$ and $w$ are two vertices in the same edge.
\item Each edge of $g$ is weighted by a map $\mathbf w\colon E_\Gamma\rightarrow \mathbb Z$, such that for every $e\in E_\Gamma$, $\mathbf w(e)>0$ and $\sum_{e\in E_\Gamma} \mathbf w(e)=d$.
\item The marks of the vertices of $g$ are given by a map $\mathbf{q}\colon A\rightarrow V_\Gamma$, where $A=\{1,\ldots, m\}$ if $m>0$, and $A=\emptyset$ if $m=0$.
\end{enumerate}
In other words, the number of all connected components of $\overline{M}_{0,m}(\mathbb{P}^{n},d)^T$ equals the number of possible tuples $\Gamma=(g,\mathbf{c},\mathbf{w},\mathbf{q})$ counted modulo isomorphism. An isomorphism between two such tuples is an isomorphism of graphs preserving $\mathbf c, \mathbf w$ and $\mathbf{q}$.

\begin{example}
Consider the space $\overline{M}_{0,0}(\mathbb{P}^{2},1)$. Since the degree is $1$, all possible graphs must have one edge by item $(3)$ and the weight of this edge will be $1$. We can use three colors, namely $x_0,x_1,x_2$, and so we have just three possible coloring (modulo isomorphism). Item $(4)$ is an empty condition since $m=0$. Hence, we have three connected components labelled by the following three graphs.
\end{example}
{
\centering
\begin {tikzpicture}[auto ,node distance =1.5 cm,on grid , thick , state/.style ={ text=black }] 
\node[state] (A1){$x_0$};  \node[state] (B1) [right =of A1] {$x_1$}; \path (A1) edge node{$1$} (B1);

\node[state] (A2)[right =of B1]{$x_0$};  \node[state] (B2) [right =of A2] {$x_2$}; \path (A2) edge node{$1$} (B2);
\node[state] (A3) [right =of B2] {$x_1$}; \node[state] (B3) [right =of A3] {$x_2$};  \path (A3) edge node{$1$} (B3);
\end{tikzpicture}
\par}
\begin{example}\label{exa:112}
For the space $\overline{M}_{0,2}(\mathbb{P}^{1},1)$, the number of connected components is four, labelled by the following graphs:

{
\centering
\begin {tikzpicture}[auto ,node distance =1.5 cm,on grid , thick , state/.style ={ text=black }] 
\node[state] (A1){$x_0$}; \node[] at (0,0.5) {$\{1,2\}$}; \node[state] (B1) [right =of A1] {$x_1$}; \path (A1) edge node{$1$} (B1);

\node[state] (A2)[right =of B1]{$x_0$}; \node[] at (3,0.5) {$\{1\}$}; \node[state] (B2) [right =of A2] {$x_1$}; \node[] at (4.5,0.5) {$\{2\}$}; \path (A2) edge node{$1$} (B2);


\node[state] (A3) [right =of B2] {$x_0$}; \node[] at (6,0.5) {$\{2\}$}; \node[state] (B3) [right =of A3] {$x_1$}; \node[] at (7.5,0.5) {$\{1\}$}; \path (A3) edge node{$1$} (B3);
\node[state] (A4)[right =of B3]{$x_0$}; \node[] at (4,0.5) {}; \node[state] (B4) [right =of A4] {$x_1$}; \node[] at (10.5,0.5) {$\{1,2\}$}; \path (A4) edge node{$1$} (B4); \end{tikzpicture}
\par}
\noindent The numbers inside the brackets are the marks of the vertices.
\end{example}
\begin{example}
For the space $\overline{M}_{0,0}(\mathbb{P}^{n},2)$, the graphs may have one or two edges by item $(3)$. It can easily be seen that we have the following possibilities.
\end{example}
{
\centering
\begin {tikzpicture}[auto ,node distance =1.5 cm,on grid , thick , state/.style ={ text=black }] 
\node[state] (A){$x_i$};
\node[state] (B) [right =of A] {$x_j$};
\path (A) edge node{$2$} (B);
\node [] at (4.5,0.1) {$0\le i < j\le n$};
\end{tikzpicture}
\par}

\smallskip
\smallskip
\smallskip
\smallskip
\smallskip

{
\centering
\begin {tikzpicture}[auto ,node distance =1.5 cm,on grid , thick , state/.style ={ text=black }] 
\node[state] (A){$x_i$};
\node[state] (B) [right =of A] {$x_j$};
\node[state] (C) [right =of B] {$x_k$};
\path (A) edge node{$1$} (B);
\path (B) edge node{$1$} (C);
\node [ ] at (5.35,0.2) {$0\le i < k\le n$};
\node [below] at (5.7,0) {$0\le j\le n,\,j\neq i,k$};

\end{tikzpicture}
\par}
The correspondence between decorated graphs and connected components of $\overline{M}_{0,m}(\mathbb{P}^{n},d)^T$ can be made explicit in the following way.
For each decorated graph $\Gamma=(g,\mathbf{c},\mathbf{w},\mathbf{q})$, we associate the substack $\overline{\mathcal{M}}_{\Gamma}$ of all marked stable maps $(C,f,p_1,\ldots,p_m)\in \overline{M}_{0,m}(\mathbb{P}^{n},d)$ such that the followings hold.
\begin{enumerate}
    \item There is a bijective correspondence between vertices $v$ of $g$ and connected components $C_v$ of $f^{-1}\{x_0,\ldots,x_n\}$. In particular, $C_v$ is either a point or a contracted union of curves. The color of $v$ is the image of $C_v$.
    \item There is a bijective correspondence between edges $e=(v,w)$ of $g$ and irreducible components of $C$ mapped to the line through $\mathbf c(v)$ and $\mathbf c(w)$, with a map of degree $\mathbf w(e)$.
    \item If $m>0$, the set of marked points $\{p_i\}_{i=1}^m$ is mapped to the set of fixed points $\{x_i\}_{i=0}^n$. We define a map $\mathbf{q}\colon\{1,\ldots,m\}\rightarrow V_\Gamma$ such that $\mathbf{q}(i)=v$ if and only if $p_i\in C_v$.
\end{enumerate}
Conversely, every connected component of the fixed locus is of the form $\overline{\mathcal{M}}_{\Gamma}$ for some decorated graph $\Gamma$.

We denote by $\mathrm{val}(v)$ the number of edges connected to the vertex $v$. Note that $C_v$ is of dimension $1$ if and only if $n(v):=\mathrm{val}(v)+|\mathbf{q}^{-1}(v)|$ satisfies $n(v)\ge 3$. Let us denote 
\begin{equation}
M_\Gamma:=\prod_{v:\dim C_v=1}\overline{M}_{0,n(v)}.\label{eq:M_Gamma}    
\end{equation}
If $|\mathrm{Aut}(\Gamma)|$ is the order of the group of automorphisms
of the decorated graph~$\Gamma$, there exists a group of order
\[
a_{\Gamma}:=|\mathrm{Aut}(\Gamma)|\cdot\prod_{e\in E_{g}}\mathbf w(e)
\]
acting on $M_\Gamma$ with quotient map given by:
$$
\phi\colon M_\Gamma\rightarrow \overline{\mathcal{M}}_{\Gamma}.
$$
For any space $X$ with a $T$-action, let us denote by $H_T^*(X)$ the equivariant cohomology of $X$. Moreover, let us denote by $\lambda_{i}$ the weight of the $T$-action on $\mathcal{O}_{\mathbb{P}^{n}}(-1)_{|x_{i}}$. Hence the equivariant cohomology of a single point is isomorphic to
$$
H_T^*(\mathrm{Spec}(\mathbb C))\cong\mathbb{C}[\lambda_{0},\ldots,\lambda_{n}].
$$
Let us denote by $\mathcal R_T$ the field of fractions of $H_T^*(\mathrm{Spec}(\mathbb C))$. By \cite[Proposition 9.1.2]{cox1999mirror} we know that there is an isomorphism
$$
H_T^*(\overline{M}_{0,m}(\mathbb{P}^{n},d))\otimes \mathcal R_T\cong \bigoplus_\Gamma H_T^*(\overline{\mathcal{M}}_{\Gamma})\otimes \mathcal R_T.
$$
From this isomorphism we deduce the Atiyah-Bott formula
\begin{equation}
\int_{\overline{M}_{0,m}(\mathbb{P}^{n},d)}P=\sum_{\Gamma}\frac{1}{a_{\Gamma}}
\frac{P^{T}(\Gamma)}{\ctopT(N_{\Gamma})(\Gamma)},\label{eq:Bott}
\end{equation}
where
\begin{itemize}
    \item $N_{\Gamma}$ is the normal bundle (as a stack) of $\overline{\mathcal{M}}_{\Gamma}$;
    \item $P$ is a polynomial in Chern classes of equivariant vector bundles of the moduli space $\overline{M}_{0,m}(\mathbb{P}^{n},d)$;
    \item $P^T$ is the equivariant polynomial of $P$, obtained by substitution of the Chern classes appearing in $P$ with the corresponding equivariant Chern classes;
    \item $\ctopT(N_\Gamma)$ is the equivariant top Chern class of $N_\Gamma$, i.e., same as before when $P=\ctop(N_\Gamma)$;
    \item $P^{T}(\Gamma)$ is the restriction to $H_T^*(\overline{\mathcal{M}}_{\Gamma})$ of $P^T$.
\end{itemize}
In other words, by Equation \ref{eq:Bott}, we can compute every integral of $P$ as a sum of elements of $\mathcal R_T$.
\subsection{Explicit formulas}
We use the following notation.
\begin{itemize}
\item For each $v\in V_{\Gamma}$ colored by $x_{i}$, the symbols $\lambda_{v}$ and $\lambda_{x_i}$ denote
$\lambda_{i}$.
\item For each $e\in E_{\Gamma}$, the symbols $d_{e}$ and $(e_{1},e_{2})$ are, respectively,
the weight of $e$ and the pair of colors of its vertices.
\item For each $v\in V_{\Gamma}$, a flag $F$ of $v$ is a pair $(v,e)$ where $e$ is an edge having $v$ as one of its vertices. Moreover if $w$ is the  other vertex of $e$, we denote by $\omega_F:=(\lambda_v-\lambda_{w})/d_{e}$, and by $F_v$ the set of flags of $v$.
\end{itemize}
All localizations of the invariants in Section $2$ can be explicitly computed. 
See \cite{kontsevich1995enumeration,cox1999mirror, muratore2020enumeration} and reference therein for the following equivariant classes (here $b>0>k$):
\begin{align}\nonumber
\ctopT(\delta_{*}\ev_{m+1}^{*}(\mathcal{O}_{\mathbb{P}^{n}}(b)))(\Gamma) & =  \prod_{e\in E_{\Gamma}}\prod_{\alpha=0}^{bd_{e}}\frac{\alpha\lambda_{e_{1}}+(bd_{e}-\alpha)\lambda_{e_{2}}}{d_{e}}\prod_{v\in V_{\Gamma}}(b\lambda_{v})^{1-\mathrm{val}(v)};\\\nonumber
\ctopT(\mathcal{E}_{d,m})(\Gamma) & =  \prod_{e\in E_{\Gamma}}\prod_{\alpha=1}^{2d_{e}-1}\frac{\alpha\lambda_{e_{1}}+(2d_{e}-\alpha)\lambda_{e_{2}}}{d_{e}}\prod_{v\in V_{\Gamma}}(2\lambda_{v})^{\mathrm{val}(v)-1};\\\nonumber
\ctopT(R^1\delta_{*}\ev_{m+1}^{*}(\mathcal{O}_{\mathbb{P}^{n}}(k)))(\Gamma) & = \prod_{e\in E_{\Gamma}}\prod_{\alpha=-1}^{kd_{e}-1}\frac{\alpha\lambda_{e_{1}}+(kd_{e}-\alpha)\lambda_{e_{2}}}{d_{e}}\prod_{v\in V_{\Gamma}}(k\lambda_{v})^{\mathrm{val}(v)-1};\\\label{eq:Incidency}
[\delta_{*}(\ev^*_{m+1}{h}^{r+1})]^{T}(\Gamma) & =  \sum_{e\in E_{\Gamma}}\sum_{t=0}^{r}d_{e}\lambda_{e_{1}}^{t}\lambda_{e_{2}}^{r-t};\\\label{eq:O1}
\ev_{j}^{*}(\mathcal{O}_{\mathbb{P}^{n}}(1)) & =  \lambda_{\mathbf{q}(j)}.
\end{align}

In order to compute the contribution from the Euler class of $N_\Gamma$ and from the $\psi$-classes, we introduce the following notation (see \cite[3.3.4]{kontsevich1995enumeration})
\begin{multline*}
X(\Gamma) := \prod_{e\in E_\Gamma}\left(\frac{(-1)^{d_e}}{(d_e!)^2}\left( \frac{d_e}{\lambda_{e_1}-\lambda_{e_2}}\right)^{2d_e}\prod_{k\neq e_1,e_2}
\prod_{a=0}^{d_e}\frac{d_e}{a\lambda_{e_1}+(d_e-a)\lambda_{e_2}-d_e\lambda_{k}}
\right)\\
\times \prod_{v\in V_\Gamma}\left(\prod_{j\neq \mathbf{c}(v)} \lambda_{v}-\lambda_{j} \right)^{\mathrm{val}(v)-1}.
\end{multline*}

Let $\{a_i\}_{i=1}^m$ be non-negative integers, and let $v$ be a vertex. 
Let $S_v:=\mathbf{q}^{-1}(v)=\{j_1,...,j_k\}$ denote the set of marks mapped to $v$, and let $\overline{S_v}:=\sum_{i\in S_v}a_i$ and $N:=n(v)-3-\overline{S_v}$. If $n(v)\ge3$ and $N\ge0$, we recall the multinomial coefficients
\begin{equation}\label{eq:multinomial}
    \binom{n(v)-3}{N,a_{j_1},...,a_{j_k}} = \frac{(n(v)-3)!}{N!\cdot a_{j_1}!\cdots a_{j_k}!}.
\end{equation}
We may extend the value of equation (\ref{eq:multinomial}) in the following way:
\begin{itemize}
    \item if $\overline{S_v}>n(v)-3\ge0$, then its value is $0$;
    \item if $S_v=\emptyset$ or $\overline{S_v}=0$, then its value is $1$;
    \item if $S_v=\{a_{j_1}\}$ and $n(v)=2$, then its value is $(-1)^{a_{j_1}}$.
\end{itemize}
As we saw in (\ref{eq:M_Gamma}), $M_\Gamma$ is a product of moduli spaces $\overline{M}_{0,n(v)}$ when $n(v)\ge3$. In each space, we have formal classes $\psi_F$ for each $F\in F_v$, and $\psi_i$ for each $i\in S_v$.

\begin{lem}\label{lem:integral}
We have the following integral in formal variables:
\begin{equation*}
\int_{\overline{M}_{0,n(v)}}\frac{\prod_{i\in S_v}\psi_i^{a_i}}{\prod_{F\in F_v}\omega_F-\psi_F} = \binom{n(v)-3}{N,a_{j_1},...,a_{j_k} }
\prod_{F\in F_v}\omega_F^{-1}\left(\sum_{F\in F_v}\omega_F^{-1} \right)^N.
\end{equation*}
\end{lem}
\begin{proof}
First we compute that
\begin{eqnarray*}
\int_{\overline{M}_{0,n(v)}}\frac{\prod_{i\in S_v}\psi_i^{a_i}}{\prod_{F\in F_v}\omega_F-\psi_F} & =&
\int_{\overline{M}_{0,n(v)}}\frac{\prod_{i\in S_v}\psi_i^{a_i}}
{\prod_{F\in F_v}\omega_F\left( 1-\frac{\psi_F}{\omega_F}\right)} \\
& =&\prod_{F\in F_v}\omega_F^{-1}\int_{\overline{M}_{0,n(v)}}\frac{\prod_{i\in S_v}\psi_i^{a_i}}
{\prod_{F\in F_v}\left( 1-\frac{\psi_F}{\omega_F}\right) } \\
& =&\prod_{F\in F_v}\omega_F^{-1} \sum_K
\int_{\overline{M}_{0,n(v)}}
 \prod_{i\in S_v}\prod_{F\in F_v}
 \psi_i^{a_i}\left(\frac{\psi_F}{\omega_F}\right)^{k_F},
\end{eqnarray*}
where the sum runs over all partitions\footnote{A partition of $N$ is an indexed tuple of non-negative integers whose sum is $N$.} $K=\{k_F\}_{F\in F_v}$ of $N$. Note that we used $1/(1-x)=1+x+x^2+...$ and the fact that the integral is zero if the sum of the exponents of all $\psi$-classes is not $n(v)-3$. Finally
\begin{align}\nonumber
\sum_K \left( \prod_{F\in F_v} \frac{1}{\omega_F^{k_F}}\right)
\int_{\overline{M}_{0,n(v)}}
 \prod_{i\in S_v}
 \psi_i^{a_i}\prod_{F\in F_v}\psi_F^{k_F}
 & =
\sum_K \left(\prod_{F\in F_v}\frac{1}{\omega_F^{k_F}} \right)
\frac{(n(v)-3)!}{\prod_{i\in S_v} a_i!\prod_{F} k_F!}\\\nonumber
 & = \frac{(n(v)-3)!}{N!\prod_{i\in S_v} a_i!}
\sum_K N!\prod_{F\in F_v}\frac{1}{k_F!\, \omega_F^{k_F}}\\\nonumber
& = \binom{n(v)-3}{N,a_{j_1},...,a_{j_k}}
\left(\sum_{F\in F_v} \frac{1}{\omega_F} \right)^N.
\end{align}
The value of the integral in the first equation follows from \cite{Witten,Airy}.
\end{proof}
In the case $n(v)\le2$, we may formally extend the integral according to (\ref{eq:multinomial}):
\begin{align*}
\int_{\overline{M}_{0,1}}\frac{1}{\omega_F-\psi_F}& =  \omega_F\\
\int_{\overline{M}_{0,2}}\frac{\psi_i^{a_{j_1}}}{\omega_F-\psi_F} & =  (-\omega_F)^{a_{j_1}}\\  
\int_{\overline{M}_{0,2}}\frac{1}{\prod_{F\in F_v}\omega_F-\psi_F} & =  
\frac{1}{\sum_{F\in F_v}\omega_F}\,\,\,(\mathrm{if} \,\,S_v=\emptyset). 
\end{align*}
The contribution from the Euler class of $N_\Gamma$ and from the $\psi$-classes is
\begin{eqnarray}\label{eq:ctopKont}
\frac{[\psi_1^{a_1}\cdots \psi_m^{a_m}]^T(\Gamma)}{\ctopT(N_{\Gamma})(\Gamma)} & = &X(\Gamma)\prod_{v\in V_\Gamma} \int_{\overline{M}_{0,n(v)}}
\frac{\prod_{i\in S_v}\psi_i^{a_i}}{\prod_{F\in F_v}\omega_F-\psi_F}.
\end{eqnarray}
This is proved in \cite[3.3]{kontsevich1995enumeration} when $a_i=0$ for $i=1,\ldots,m$
(see subsections 3.3.3 and 3.3.5 for the explicit formulas). For the general case, we refer to \cite[3.7]{leepand} or \cite[Section 4]{LiuSheshmani}. We use Lemma \ref{lem:integral} to compute the right-hand side of (\ref{eq:ctopKont}).


In the code, in order to have the $\psi$-classes as an \textit{external} function independent of $\ctopT(N_{\Gamma})(\Gamma)$, we implemented the following two functions:
\begin{eqnarray}\label{eq:Euler_inv}
\mathrm{Euler\_inv} & := & X(\Gamma)\prod_{v\in V_\Gamma} \int_{\overline{M}_{0,n(v)}}\frac{1}{\prod_{F\in F_v}\omega_F-\psi_F} \nonumber \\
 & = & X(\Gamma)\prod_{v\in V_\Gamma}\prod_{F\in F_v}\omega_F^{-1}\left(\sum_{F\in F_v}\omega_F^{-1} \right)^{n(v)-3}\\\label{eq:Psi}
\mathrm{Psi} & := &\prod_{v\in V_\Gamma}
\int_{\overline{M}_{0,n(v)}}\frac{\prod_{i\in S_v}\psi_i^{a_i}}{\prod_{F\in F_v}\omega_F-\psi_F} \left( \int_{\overline{M}_{0,n(v)}}\frac{1}{\prod_{F\in F_v}\omega_F-\psi_F}\right)^{-1}\nonumber \\
 & =& \prod_{v\in V_\Gamma} \binom{n(v)-3}{n(v)-3-\overline{S_v},a_{j_1},...,a_{j_k} }
\left(\sum_{F\in F_v}\omega_F^{-1} \right)^{-\overline{S_v}}.
\end{eqnarray}
\begin{example}
Suppose we want to compute
$$
I = \int_{\overline{M}_{0,0}(\mathbb{P}^{2},1)}\delta_{*}(\ev_1^*(h^{2}))\cdot\delta_{*}(\ev_1^*(h^{2})).
$$

\noindent When $\Gamma$ is

{
\centering
\begin {tikzpicture}[auto ,node distance =1.5 cm,on grid , thick , state/.style ={ text=black }] 
\node[state] (A1){$x_i$};  \node[state] (B1) [right =of A1] {$x_j$}; \path (A1) edge node{$1$} (B1);
\node [] at (4.5,0.1) {$0\le i < j\le 2,$};
\end{tikzpicture}
\par}
\noindent{we know that}

\begin{eqnarray*}
\ensuremath{[\delta_{*}(\ev_{1}^{*}h^{2})]^{T}(\Gamma)} & = & \lambda_{i}+\lambda_{j}\\
\frac{1}{\ctopT(N_{\Gamma})(\Gamma)} & = & \frac{1}{\lambda_{j}-\lambda_{k}}\frac{1}{\lambda_{i}-\lambda_{k}},
\end{eqnarray*}
where $k\in\{0,1,2\}\diagdown\{i,j\}$.
By Example \ref{exa:112}, since $\overline{M}_{0,0}(\mathbb{P}^{2},1)^T$ has three connected components, the application of the Atiyah-Bott formula will express $I$ as the sum of three rational polynomials:

\[
I  =  \frac{\lambda_{0}+\lambda_{1}}{\lambda_{1}-\lambda_{2}}\frac{\lambda_{0}+\lambda_{1}}{\lambda_{0}-\lambda_{2}}+\frac{\lambda_{0}+\lambda_{2}}{\lambda_{2}-\lambda_{1}}\frac{\lambda_{0}+\lambda_{2}}{\lambda_{0}-\lambda_{1}}+\frac{\lambda_{1}+\lambda_{2}}{\lambda_{2}-\lambda_{0}}\frac{\lambda_{1}+\lambda_{2}}{\lambda_{1}-\lambda_{0}}
  =  1.
 \]
\end{example}
\section{The Algorithm}
\begin{algorithm}[htbp]\label{alg:ABF}
  \caption{\textsc{The computation of the Atiyah--Bott Formula}}
    \LinesNumbered
  \SetKw{Set}{Set}
  \KwIn{$\mathrm{P:}$ Equivariant class\\
  \hspace{1.2cm} $n$, $d$, $m$}
  \KwOut{The degree of P in $\overline M_{0,m}(\mathbb P^n,d)$}
  \BlankLine
  \Set{$\mathrm{s:}$ array of $n+1$ random 
        rationals between $0$ and $2^{32}-1$}.\\
	\Set{$\mathrm{ans}=0$}\\
  \For{$\mathrm{g}$ a tree with $2,\ldots,d+1$ vertices}
  {
    \For{$\mathrm{c}$ a coloring of $\mathrm{g}$ with at most $n+1$ colors}
    {
        \Set{$\mbox{aut}=$ the number of automorphisms of $\mathrm{g}$ preserving $\mathrm{c}$} \\
        \For{$\mathrm{q}$ a marking of $\mathrm{g}$ with $m$ marks}
        {
            \For{$\mathrm{w}$ a set of edge weights of $\mathrm{g}$}
            {   
                \Set{$\mathrm{ans} = \mathrm{ans} + \mathrm{P(g,c,w,s,q)*Euler\_inv(g,c,w,s,q)}/(\mbox{aut}*\mbox{prod}(\mathrm{w}))$}
        }
    }
}
}
\KwRet{$\mathrm{ans}$}  
  \end{algorithm}
  
Algorithm~\ref{alg:ABF} describes the computation of the 
equivariant class \verb!P! in the moduli space
$\overline M_{0,m}(\mathbb P^n,d)$. 
The for loop starting at line~3 runs through the isomorphisms classes of tree 
graphs with at least $2$ and at most $d+1$ vertices.
The list of such graphs with up to $14$ vertices was
precomputed using the SageMath computational algebra system and stored in a file as a list of Pr{\"u}fer sequences\footnote{Given a tree $g$ of $k$ vertices and a bijective map $f:V_\Gamma\rightarrow \{1,2,...,k\}$, a Pr{\"u}fer sequence is a sequence of $k-2$ numbers in $\{1,2,...,k\}$, which recovers $g$ and $f$ modulo automorphisms of $g$.}.
Then the for loop starting from line~4 runs through all 
the equivalence classes of colorings of the current value of $g$. These colorings are also precomputed with SageMath and stored in a file as a list of arrays of integers. The variable $aut$ in line~5 
will hold the size of the automorphisms group of the colored graph $g$ with 
respect to the coloring $c$. Then the loops starting from lines 6 and 7 run through all possible markings and edge weights without considering equivalence. The function $prod(\mathrm{w})$ denotes the product of all weights.

Our implementation of an equivariant class is a Julia function, returning a rational number, with arguments including \verb!g,c,w,s,q!, such that:
\begin{itemize}
    \item \verb!g! is a graph implemented using the Julia package Graphs~\cite{Graphs2021}. It is generated from a Pr{\"u}fer sequence, which gives an enumeration of the vertices of \verb!g!.
    \item \verb!c! is an array of integers of length equals to the number of vertices of \verb!g!, the $i^\mathrm{th}$ entry of \verb!c! is the color of the $i^\mathrm{th}$ vertex.
    \item \verb!w! is an array of integers of length equals to the number of edges of \verb!g!, whose sum is $d$.
    \item \verb!s! is an array of random rational numbers of length $n+1$, generated at the beginning of the computation.
    \item \verb!q! is an array of integers of length $m$, the $i^\mathrm{th}$ entry of \verb!q! is the vertex containing the $i^\mathrm{th}$ mark.
\end{itemize}

Not all functions need the argument \verb!q!, and some functions need additional arguments. The symbol $\lambda_v$ is implemented as \verb!s[c[v]]!. 
Here \verb!v! is the number of a vertex, so \verb!c[v]! is the color of that vertex, and finally \verb!s[c[v]]! is a rational number. On the other hand, $\lambda_{\mathbf{q}(j)}$ is implemented in the obvious way. For example, the equivariant class (\ref{eq:O1}) is
\begin{verbatim}
function O1_i(g,c,w,s,q,j) 
#O1_i is the name of the function
#j is the number of the mark
    return s[c[q[j]]];
end
\end{verbatim}
This is the simplest equivariant class that we can define. All other equivariant classes can be defined using the operations of the rational numbers. Those operations are well defined in the entries of the array \verb!s!. 

For each edge we associate an entry of \verb!w! using the structure \textit{Dictionary} provided by the Julia language. The two vertices of an edge $e$ are denoted, in Julia, by \verb!src(e)! and \verb!dst(e)!.

The equivariant class (\ref{eq:Incidency}) is the following (\verb!exp! is the exponent of h).
\begin{verbatim}
function Incidency(g,c,w,s,exp)
    p1 = fmpq(0); #the final result
    #it is initialised as the rational number 0/1
    
    r = exp-1;                          
    d = Dict(edges(g).=> w); #assign weights to edges
    #the weight of any edge e is d[e]

    for e in edges(g)
        for t in 0:r
            p1 = p1 + d[e]*(s[c[src(e)]]^(t))*(s[c[dst(e)]]^(r-t));
        end
    end

    return p1;
end
\end{verbatim}

We implemented all equivariant classes listed in Section \ref{sec:3} following the same strategy.
This is possible since those classes are written in explicit combinatorial terms. Even if some of them seem complicated, they are relatively easy to implement in Julia.

The function \verb!P!, defined by the user, is any combination of Julia functions implementing equivariant classes, see the next section for some examples. 
The function \verb!Euler_inv(g,c,w,s,q)!, defined in (\ref{eq:Euler_inv}), coincides with the contribution \eqref{eq:ctopKont} if no $\psi$-classes are involved. 
Otherwise, \verb!Psi(g,c,w,s,q,a)!, defined in (\ref{eq:Psi}), corrects this discrepancy (see (\ref{eq:dualc2}) for an example). The argument \verb!a! in the definition of the function \verb!Psi! is an array of integers representing the exponents of the $\psi$-classes. 

After each loop of line 8 the value of
$$
\mathrm{P(g,c,w,s,q)}*\frac{\mathrm{Euler\_inv(g,c,w,s,q)}}{aut*prod(\mathrm{w})},
$$
which depends on \verb!s!, is added to \verb!ans!. The value of \verb!ans! after the completion of all the loops does not depend on \verb!s!.
\begin{rem}
As we said before, we consider $aut$ as the order of the group of automorphisms of \verb!g! preserving \verb!c!, and not those preserving the full decorated graph (that is, including \verb!w! and \verb!q!) as in Eq.~\eqref{eq:Bott}. This because we would have to deal with a much bigger database of decoreted graphs. The computation could have been faster, but with a high impact on user's hard-disk. However this does not change the final result since we run the loops for all possible weights and markings. 
\end{rem}
\begin{rem}
We use the Nemo package \cite{nemo} for rational number arithmetic instead of the much slower Julia's built-in \verb!Rational{BigInt}!. We thank Jieao Song for suggesting us this improvement.
\end{rem}
\begin{rem}
The current version of the package is 2.0.0. In the future, some parts of the code may be improved. We refer to the internal documentation of the package for an up-to-date explanation of its features.
\end{rem}
\section{Explicit Computations}
An implementation of the computational methods described in this paper is available in the Julia language. 
The implementation can be obtained from the repository {\tt https://github.com/mgemath/AtiyahBott.jl}.
To use our code, the user should first define the equivariant classes to be calculated. For example, the user might type
\begin{verbatim}
P = Hypersurface(5);  
\end{verbatim}
The supported equavariant classes are listed in the file \verb!EquivariantClasses.jl! (see also 
below for some explicit examples). After \verb!P! is defined, one has to call the Atiyah-Bott formula by the command
\begin{verbatim}
AtiyahBottFormula(n,d,m,P);
\end{verbatim}
This function will calculate \verb!P! in the moduli space $\overline{M}_{0,m}(\mathbb{P}^{n},d)$, as explained in the previous section.
For multiple computations in the same moduli space, the variable \verb!P! can be an array of equivariant classes. This way, the computation can be done simultaneously for several equivariant classes
and this reduces the running time. 

More detailed documentation of the package can be found in the github repository referred to above. 
Some examples with timing (in seconds) are given in the rest of the section. All computations were run on an Intel$^\circledR$ Xeon$^\circledR$ E-2244G CPU running at 3.80GHz with 64GB of memory and 8 threads, using Version 2.0.0 of the package.

This program can support any polynomial in equivariant classes of vector bundles over $\overline{M}_{0,m}(\mathbb{P}^{n},d)$, as long as the problem can be translated to a graph-theoretic computation, as in the examples below. In case 
the user needs to add more such classes, the authors are available for support.

\subsection{Projective spaces}
\subsubsection{Plane curves}
We know that the invariant $$\int_{\overline{M}_{0,3d-1}(\mathbb{P}^{2},d)} \prod_{i=1}^{3d-1} \ev_i^*(h)^2$$
equals the number of rational plane curves of degree $d$ through $3d-1$ general points. 
The computation of this invariant can be performed with our package as follows: 
\begin{verbatim}
d = 1; #for other values of d, change this line
P = O1()^2;
AtiyahBottFormula(2,d,3*d-1,P);
\end{verbatim}
The results with running times are listed in Table \ref{tab:planecurves}. This particular computation is not optimal since, as one can deduce from Example \ref{exa:112}, the marks make the number of connected components of $\overline{M}_{0,3d-1}(\mathbb{P}^{2},d)^T$ much bigger than $\overline{M}_{0,0}(\mathbb{P}^{2},d)^T$.
Using Remark \ref{rem:semplif} (that is, the case $m=0$), we can drastically reduce the number of decorated graphs, 
speeding up the computation. Whenever possible, we will adopt this simplification in the forthcoming examples.
\begin{table}[htbp]
\caption{Numbers of rational plane curves.}
    \centering
    \begin{tabular}{c|c|c}
    $d$ & value & runtime \\
    \hline
        1 & 1 &  0.00661\mbox{ s}  \\
        2 & 1 &   0.0259\mbox{ s}\\
        3 & 12 &   13.9\mbox{ s}\\
    \end{tabular}
    
    \label{tab:planecurves}
\end{table}

\subsubsection{Contact curves}
By the previous section, we know that the number of rational contact curves in $\mathbb P^3$ meeting $a$ general points and $b=2(d-a)+1$ general lines is
\[
\int_{\overline{M}_{0,0}(\mathbb{P}^{3},d)}\delta_*(\ev_{1}^{*}(h^{3}))^a\cdot\delta_*(\ev_{1}^{*}(h^{2}))^b\cdot\ctop(\mathcal{E}_{d,0}).
\]
This  invariant can be computed as follows:
\begin{verbatim}
d = 1; #for other values of d, change this line
P = [ Incidency(3)^a*Incidency(2)^(2*(d-a)+1)*Contact() 
            for a in 0:d ];
AtiyahBottFormula(3,d,0,P);
\end{verbatim}
In Table \ref{tab:contact} we list the values up to degree $8$. The values of the table for $a=0$ and $d\le 4$ where computed in \cite{levVa,Eden}. All values for $d\le 3$ where computed in \cite{Muratore_2021}.
\begin{table}[htbp]
\caption{Numbers of rational contact space curves.}
    \centering
    \begin{tabular}{c|c|c|c||c|c|c|c}
    $d$ & a & value & runtime& $d$ & a & value & runtime\\\hline
        1 &0  & 2  & 0.0006\mbox{ s}  &5  &0  & 539504640 & 0.4715\mbox{ s}\\
         & 1 & 1  &   &   & 1 &  37318656 &   \\ \cline{1-4}
        2 &0  & 40  & 0.003\mbox{ s}  &   &2  &  2780160 &  \\
         &1  &  8 &   &   & 3 & 224896  & \\
         &2  &  2 &   &   & 4 & 20000  & \\ \cline{1-4}
        3 &0  & 4160  &0.049\mbox{ s}   &  &5  & 2000  & \\\cline{5-8}
         &1  &  512 &   &6   & 0 &  434591612928 & 2.299 \mbox{ s}  \\
         &2  &  72 &   &   & 1 & 24641224704  & \\
         &3  & 12  &   &   & 2 & 1484648448  &  \\\cline{1-4}
        4 &0  & 1089024   & 0.323\mbox{ s}  &   & 3 & 95545344  &  \\
         &1   &  96512 &   &   &4  & 6614016  & \\  
         &2  &  9408 &   &   &  5&  497664  &\\
         &3  &  1024 &   &   &  6& 41472  & \\\cline{5-8}
         &4  &  128 &   &8  &0 & 858416911346565120  & 164.9\mbox{ s} \\ \cline{1-4}
         7 &0 & 518244677713920 & 16.02\mbox{ s} &  & 1 &  35755655304314880&\\
         & 1&   24890475282432& & & 2  & 1557262876803072& \\ 
         & 2 & 1258547527680  &  &&   3& 71102066589696 & \\
         &    3 & 67234406400 & &&    4&  3414199762944& \\
         &    4 &  3812150272 && &  5&  173116293120& \\
         &   5  & 230814720 & & &  6& 9319309312 & \\
         &    6 & 15059072 & & &7& 536870912 & \\
         &     7&  1075648 && &8&  33554432& \\
    \end{tabular}


    \label{tab:contact}
\end{table}
\subsubsection{Tangent curves}
Let $Y$ be a smooth plane curve of degree $z$. The number of lines tangent to $Y$ and passing through a general point is classically known to be $z(z-1)$. This can be proved using Grassmannians (see \cite[Section 3.6.2]{3264}). Let us prove it using the moduli space of stable curves.
Following \cite[Definition 1.1]{gath}, we denote by 
$\overline{M}_{(2,0)}^Y(\mathbb{P}^{2},1)$ the moduli space of stable maps in $\overline{M}_{0,2}(\mathbb{P}^{2},1)$ tangent to $Y$ at the first marked points. If $x\in \mathbb{P}^2$ is a general point, we have

\begin{eqnarray}
z(z-1) & = & \int_{[\overline{M}_{(2,0)}^Y(\mathbb{P}^{2},1)]^{\mathrm{virt}}}\ev_{2}^{*}(x)\label{eq:firstdualc} \\
 & = & \int_{\overline{M}_{0,2}(\mathbb{P}^{2},1)}\ctop{(\mathrm{J}^{1}(\ev_{1}^{*}\mathcal{O}_{\mathbb{P}^{2}}(z)))}\cdot\ev_{2}^{*}(h^{2})\label{eq:dualc}
\end{eqnarray}

Equation \ref{eq:firstdualc} is an application of Kleiman-Bertini \cite{Kle}: it is well-know that $\overline{M}_{(2,0)}^Y(\mathbb{P}^{2},1)$ is of the expected dimension (in particular, its virtual fundamental class equals its usual fundamental class), thus for a general $x$ the subspaces $\overline{M}_{(2,0)}^Y(\mathbb{P}^{2},1)$ and $\ev_{2}^{-1}(x)$ meet transversely in a zero cycle. Equation \ref{eq:dualc} is deduced from \cite[Theorem 2.6]{gath}.

Using \eqref{eq:PsiGW_std} and Remark \ref{rem:jet} with $p=1$:

\begin{equation*}
0\rightarrow\ev_1^*\mathcal{O}_{\mathbb{P}^n}(z)\otimes\mathbb L^{}_1 \rightarrow
\mathrm{J}^{1}(\ev_1^*\mathcal{O}_{\mathbb{P}^n}(z))\rightarrow
\ev_1^*\mathcal{O}_{\mathbb{P}^n}(z)\rightarrow 0,
\end{equation*}
the integral in Equation \eqref{eq:dualc} can be expanded as
\begin{equation}\label{eq:dualc2}
\int_{\overline{M}_{0,2}(\mathbb{P}^{2},1)}
z\ev_1^*(h)\cdot(z\ev_1^*(h)+\psi_1\psi_2^0)\cdot\ev_2^*(h^{2}).
\end{equation}
It can be computed with the following instructions:
\begin{verbatim}
z = 1; #for other values of z, change this line
P = z*O1_i(1)*(z*O1_i(1)+Psi([1,0]))*O1_i(2)^2;
AtiyahBottFormula(2,1,2,P);
\end{verbatim}

The number of flex tangents of a smooth plane curve of degree $z$ is exactly
\begin{equation}\label{eq:flex}
3z(z-2)=\int_{\overline{M}_{0,1}(\mathbb{P}^{2},1)}
\ctop{(\mathrm{J}^2(\ev_1^*\mathcal{O}_{\mathbb{P}^2}(z)))},
\end{equation}
see \cite[Section 11.3]{3264}.
We can compute it either by expanding the Euler class of the jet bundle, or by using the function \verb'Jet' provided by the package:
\begin{verbatim}
z = 3; #for other values of z, change this line
P = Jet(2,z);
AtiyahBottFormula(2,1,1,P);
\end{verbatim}

Other interesting invariants can be deduced using the jet bundle. For example,
\begin{equation}\label{eq:oscu}
\int_{\overline{M}_{0,1}(\mathbb{P}^{n},d)}
\ev_1^*(h^{n-1})\cdot\ctop{(\mathrm{J}^{(n+1)d-2}(\ev_1^*\mathcal{O}_{\mathbb{P}^m}(1)))}
\end{equation}
is an invariant related to the osculating curves of a hypersurface in $\mathbb P^n$, see \cite{Muratore_2021} and references therein. It can be calculated with

\begin{verbatim}
n = 3;d = 1; #for other values of n and d, change this line
P = O1()^(n-1)*Jet((n+1)*d-2,1);
AtiyahBottFormula(n,d,1,P);
\end{verbatim}
\newpage
\subsection{Calabi-Yau Threefolds}
\subsubsection{Quintic}

Let $X\subset\mathbb{P}^{4}$ be a general quintic hypersurface. We may define the Gromov-Witten invariants as
\begin{eqnarray*}
N_{d} & := & \int_{\overline{M}_{0,0}(\mathbb{P}^{4},d)}\ctop(\delta_{*}(\ev_{1}^{*}\mathcal{O}_{\mathbb P^4}(5))),
\end{eqnarray*}
and recursively the instanton numbers $n_{d}$ by the formula
\begin{eqnarray*}
N_{d} & = & \sum_{k|d}n_{\frac{d}{k}}k^{-3}.
\end{eqnarray*}
It was proved, for $d\leq 9$, by \cite{JohnsenKleiman} that $n_{d}$ equals the number
of rational curves in $X$ of degree $d$. All values of $n_{d}$ are
predicted by Mirror Symmetry (see \cite{Candelas}). In Table \ref{tab:quintic} we list the values of $N_{d}$ up to $d=9$. We used the following code.
\begin{verbatim}
d = 1; #for other values of d, change this line
P = Hypersurface(5);
AtiyahBottFormula(4,d,0,P);
\end{verbatim}
Kontsevich~\cite{kontsevich1995enumeration} reported that the computation for 
$N_4$ took 5 minutes on a Sun computer in 1995 -- now we can do it in a fraction of a second. 
\begin{table}[htbp]
    \caption{GW invariants of a quintic threefold.}
    \centering
    \begin{tabular}{c|c|c}
        $d$ &  value & runtime \\ \hline
        1 & 2875 & 0.0004\mbox{ s} \\
        2 & 4876875/8 & 0.00145\mbox{ s}\\
        3 & 8564575000/27& 0.004\mbox{ s}\\
        4 & 15517926796875/64 & 0.028\mbox{ s} \\
        5 &229305888887648 & 0.243\mbox{ s} \\
        6 & 248249742157695375 & 2.01\mbox{ s}\\
        7 & 101216230345800061125625/343 & 33.1\mbox{ s}\\
        8 &  192323666400003538944396875/512 & 377\mbox{ s}\\
        9 & 367299732093982242625847031250/729 & 4239\mbox{ s}
    \end{tabular}
    
    \label{tab:quintic}
\end{table}
In order to get the real number of rational curves, we can apply Manin's formula (\ref{eq:manin}) by removing from $N_d$ the contribution of the covers of the curves of lower degree. Such a contribution can be obtained by: 
\begin{verbatim}
d = 1; #for other values of d, change this line
P = R1(1)^2;
AtiyahBottFormula(1,d,0,P);
\end{verbatim}
These numbers were computed using these methods up to degree $6$ by \cite{Hiep}. Moreover, they have been computed using the original physical methods for large $d$, see \cite[A060041]{oeis} 
and reference therein.
\subsubsection{Other threefolds}
By the adjunction formula, it can easily be seen that there are four other
possibilities of complete intersection Calabi-Yau threefolds other
than the quintic hypersurface in $\mathbb{P}^4$. These are complete intersections of multidegree $(3,3)$
and $(4,2)$ in $\mathbb{P}^{5}$, $(3,2,2)$ in $\mathbb{P}^{6}$
and $(2,2,2,2)$ in $\mathbb{P}^{7}$. We computed some of the Gromov-Witten
numbers of these varieties in Table \ref{tab:CY}. Our results agree with \cite{LibTei}. See also \cite{AlbanoKatz,Berny} and references therein.
\begin{table}[htbp]
\caption{GW invariants of Calabi-Yau threefolds.}
    \centering
\begin{tabular}{c|c|c||c|c|c}
\multicolumn{3}{c||}{$(3,3)$} & \multicolumn{3}{c}{$(4,2)$}\tabularnewline
\hline 
$d$ & value & runtime & $d$ & value & runtime\tabularnewline
\hline 
 1&1053  &  0.0004\mbox{ s}&1  & 1280 &0.0007\mbox{ s} \tabularnewline
 2&  423549/8&0.0023\mbox{ s}  &2  &92448  &0.0025\mbox{ s} \tabularnewline
 3&  6424365& 0.0103\mbox{ s} &3  &422690816/27  &0.0112\mbox{ s} \tabularnewline
 4&  72925120125/64&0.1019\mbox{ s}  &4  &3883914084  &0.8888\mbox{ s} \tabularnewline
\hline 

\multicolumn{3}{c||}{$(3,2,2)$} & \multicolumn{3}{c}{$(2,2,2,2)$}\tabularnewline
\hline 
$d$ & value & runtime & $d$ & value & runtime\tabularnewline
\hline 
 1& 720&  0.0005\mbox{ s}&1  &512  & 0.0009\mbox{ s}\tabularnewline
 2&  22518&0.0037\mbox{ s}  &2  &9792  &0.0056\mbox{ s}\tabularnewline
 3&  4834592/3& 0.0252\mbox{ s} &3  &11239424/27  &0.0455\mbox{ s}\tabularnewline
 4&  672808059/4& 0.3669\mbox{ s}  &4  &25705160  & 0.7849\mbox{ s}\tabularnewline
 
\end{tabular}
    
    \label{tab:CY}
\end{table}

\subsection{Fano varieties}
\subsubsection{Cubic surface}
Let $X\subset\mathbb{P}^{3}$ be a general cubic surface. Complete recursion formulae for
the enumerative numbers of $X$ are known from \cite{DiFrancesco},
where all of them are listed up to degree $6$. It is known that the
number
$$
\int_{\overline{M}_{0,d-1}(\mathbb{P}^{3},d)}\ev_{1}^{*}(h^{2})\cdots\ev_{d-1}^{*}(h^{2})\cdot\ctop(\delta_{*}(\ev_{d}^{*}\mathcal{O}_{\mathbb P^3}(3)))
$$
is enumerative. Moreover, a general line in $\mathbb{P}^{3}$ meets
$X$ at three points, so if we want to compute the number of curves
in $X$ through $d-1$ general points, we need to multiply the above
display by $3^{1-d}$. In Table \ref{tab:cubic} we listed the number
of rational curves of degree $d$ on $X$ meeting $d-1$ general points up to degree $8$. Our results
are all in agreement with existing results.
\begin{table}[htbp]
\caption{Numbers curves in a cubic surface.}
    \centering
    \begin{tabular}{c|c|c||c|c|c}
    $d$ & value & runtime & $d$ & value & runtime \\
    \hline
        1 & 27 & 0.0003\mbox{ s}& 5 &5616  &0.0538\mbox{ s}  \\
        2 & 27 &0.0008 \mbox{ s}& 6 &82944  &0.287\mbox{ s}  \\
        3 & 84 & 0.0019\mbox{ s}& 7 &1608768  &3.9125\mbox{ s}  \\
        4 & 540 & 0.0082\mbox{ s}& 8 & 38928384 &36\mbox{ s}    \\
    \end{tabular}
    
    \label{tab:cubic}
\end{table}
\subsubsection{Quadric fourfold}
Complete recursion formulae for the quadric surfaces and threefolds are in \cite{DiFrancesco,FP}. In Table 
\ref{tab:quadric2} we listed the values of the invariant
$$
\int_{\overline{M}_{0,0}(\mathbb{P}^{5},d)}\delta_{*}(\ev_{1}^{*}(h^{4}))^a\cdot\delta_{*}(\ev_{1}^{*}(h^{3}))^b\cdot\delta_{*}(\ev_{1}^{*}(h^{2}))^c\cdot\ctop(\delta_{*}(\ev_{1}^{*}\mathcal{O}_{\mathbb P^5}(2))),
$$
where $3a+2b+c=4d+1$ for $d=1$ and $d=2$. Both tables required far less than a second of CPU time.
\begin{table}[htbp]
\caption{Numbers of lines and conics in a quadric fourfold.}
    \centering
    \,\,\,\,\,\,\begin{tabular}{c|c||c|c}
        $(a,b,c)$ & value & $(a,b,c)$ & value \\\hline
        $(0,0,5)$ & 20 & $(0,2,1)$ & 8 \\
        $(0,1,3)$ & 12 & $(1,1,0)$ & 4 \\
        $(1,0,2)$ & 4 & \\
    \end{tabular}
    \begin{tabular}{c|c||c|c}\hline
        $(a,b,c)$ & value & $(a,b,c)$ & value   \\\hline
        $(0,0,11)$ & 338878012462176 & $(0,4,3)$ & 2868685444352\\
        $(0,1,9)$ & 101436360044288 & $(0,5,1)$ & 866614157376\\
        $(0,2,7)$ & 30581975310256 & $(1,0,8)$ & 20679736185808\\
        $(0,3,5)$ & 9330647718128 & $(1,1,6)$ & 6274105265472\\

        $(1,2,4)$ & 1958615970176 & $(2,1,3)$ & 422160042016  \\
        $(1,3,2)$ & 617652337904 & $(2,2,1)$ & 139428589504  \\
        $(1,4,0)$ & 185149566080 & $(3,0,2)$ & 103721819712  \\
        $(2,0,5)$ & 1260340192768 & $(3,1,0)$ & 35140354848  \\

    \end{tabular}
    
    \label{tab:quadric2}
\end{table}
\newpage


\begin{thebibliography}{CdlOGP92}

\bibitem[AB84]{AB}
M.~F. Atiyah and R.~Bott, \emph{The moment map and equivariant cohomology},
  Topology \textbf{23} (1984), no.~1, 1--28. \MR{721448}

\bibitem[AK91]{AlbanoKatz}
Alberto Albano and Sheldon Katz, \emph{van {G}eemen's families of lines on
  special quintic threefolds}, Manuscripta Math. \textbf{70} (1991), no.~2,
  183--188. \MR{1085631}

\bibitem[AM93]{aspinwall1993topological}
Paul~S. Aspinwall and David~R. Morrison, \emph{Topological field theory and
  rational curves}, Comm. Math. Phys. \textbf{151} (1993), no.~2, 245--262.
  \MR{1204770}

\bibitem[Amo14]{Eden}
{\'E}den Amorim, \emph{Curvas de contato no espa{\c{c}}o projetivo}, Ph.D.
  thesis, Universidade Federal de Minas Gerais, Belo Horizonte, Brazil, 2014,
  In English: {\it Contact curves in the projective space}, arXiv:1907.03973.

\bibitem[BEKS17]{bezanson2017julia}
Jeff Bezanson, Alan Edelman, Stefan Karpinski, and Viral~B. Shah, \emph{Julia:
  a fresh approach to numerical computing}, SIAM Rev. \textbf{59} (2017),
  no.~1, 65--98. \MR{3605826}

\bibitem[Ber08]{Berny}
M.~Bernardara, \emph{Calabi-{Y}au complete intersections with infinitely many
  lines}, Rend. Semin. Mat. Univ. Politec. Torino \textbf{66} (2008), no.~2,
  87--97. \MR{2455452}

\bibitem[Bot67]{bott1967residue}
Raoul Bott, \emph{A residue formula for holomorphic vector-fields}, J.
  Differential Geometry \textbf{1} (1967), 311--330. \MR{232405}

\bibitem[CdlOGP92]{Candelas}
Philip Candelas, Xenia~C. de~la Ossa, Paul~S. Green, and Linda Parkes, \emph{A
  pair of {C}alabi-{Y}au manifolds as an exactly soluble superconformal
  theory}, Essays on mirror manifolds, Int. Press, Hong Kong, 1992, pp.~31--95.
  \MR{1191420}

\bibitem[CK99]{cox1999mirror}
David~A. Cox and Sheldon Katz, \emph{Mirror symmetry and algebraic geometry},
  Mathematical Surveys and Monographs, vol.~68, American Mathematical Society,
  Providence, RI, 1999. \MR{1677117}

\bibitem[DFI95]{DiFrancesco}
P.~Di~Francesco and C.~Itzykson, \emph{Quantum intersection rings}, The moduli
  space of curves ({T}exel {I}sland, 1994), Progr. Math., vol. 129,
  Birkh\"{a}user Boston, Boston, MA, 1995, pp.~81--148. \MR{1363054}

\bibitem[EH16]{3264}
David Eisenbud and Joe Harris, \emph{3264 and all that---a second course in
  algebraic geometry}, Cambridge University Press, Cambridge, 2016.
  \MR{3617981}

\bibitem[ES96]{ES}
Geir Ellingsrud and Stein~Arild Str{\o}mme, \emph{Bott's formula and
  enumerative geometry}, J. Amer. Math. Soc. \textbf{9} (1996), no.~1,
  175--193. \MR{1317230}

\bibitem[FBS{\etalchar{+}}21]{Graphs2021}
James Fairbanks, Mathieu Besan\c{c}on, Simon Sch\"{o}lly, J\'{u}lio Hoffiman,
  Nick Eubank, and Stefan Karpinski, \emph{{JuliaGraphs/Graphs.jl: an optimized
  graphs package for the Julia programming language}}, 2021.

\bibitem[FHHJ17]{nemo}
Claus Fieker, William Hart, Tommy Hofmann, and Fredrik Johansson,
  \emph{Nemo/{H}ecke: computer algebra and number theory packages for the
  {J}ulia programming language}, I{SSAC}'17---{P}roceedings of the 2017 {ACM}
  {I}nternational {S}ymposium on {S}ymbolic and {A}lgebraic {C}omputation, ACM,
  New York, 2017, pp.~157--164. \MR{3703682}

\bibitem[FP97]{FP}
W.~Fulton and R.~Pandharipande, \emph{Notes on stable maps and quantum
  cohomology}, Algebraic geometry---{S}anta {C}ruz 1995, Proc. Sympos. Pure
  Math., vol.~62, Amer. Math. Soc., Providence, RI, 1997, pp.~45--96.
  \MR{1492534}

\bibitem[Gat02]{gath}
Andreas Gathmann, \emph{Absolute and relative {G}romov-{W}itten invariants of
  very ample hypersurfaces}, Duke Math. J. \textbf{115} (2002), no.~2,
  171--203. \MR{1944571}

\bibitem[GP99]{graber1999localization}
T.~Graber and R.~Pandharipande, \emph{Localization of virtual classes}, Invent.
  Math. \textbf{135} (1999), no.~2, 487--518. \MR{1666787}

\bibitem[Hie16]{Hiep}
Dang~Tuan Hiep, \emph{Rational curves on {C}alabi-{Y}au threefolds: verifying
  mirror symmetry predictions}, J. Symbolic Comput. \textbf{76} (2016), 65--83.
  \MR{3461259}

\bibitem[JK96]{JohnsenKleiman}
Trygve Johnsen and Steven~L. Kleiman, \emph{Rational curves of degree at most
  {$9$} on a general quintic threefold}, Comm. Algebra \textbf{24} (1996),
  no.~8, 2721--2753. \MR{1393281}

\bibitem[Kle74]{Kle}
Steven~L. Kleiman, \emph{The transversality of a general translate}, Compositio
  Math. \textbf{28} (1974), 287--297. \MR{360616}

\bibitem[Kob59]{kob}
Shoshichi Kobayashi, \emph{Remarks on complex contact manifolds}, Proc. Amer.
  Math. Soc. \textbf{10} (1959), 164--167. \MR{111061}

\bibitem[Kon92]{Airy}
Maxim Kontsevich, \emph{Intersection theory on the moduli space of curves and
  the matrix {A}iry function}, Comm. Math. Phys. \textbf{147} (1992), no.~1,
  1--23. \MR{1171758}

\bibitem[Kon95]{kontsevich1995enumeration}
\bysame, \emph{Enumeration of rational curves via torus actions}, The moduli
  space of curves ({T}exel {I}sland, 1994), Progr. Math., vol. 129,
  Birkh\"{a}user Boston, Boston, MA, 1995, pp.~335--368. \MR{1363062}

\bibitem[LP04]{leepand}
Y.-P. Lee and Rahul Pandharipande, \emph{Frobenius manifolds, {G}romov-{W}itten
  theory, and {V}irasoro constraints}, 2004, available at
  www.math.utah.edu/{$\sim$}yplee/research.

\bibitem[LS17]{LiuSheshmani}
Chiu-Chu~Melissa Liu and Artan Sheshmani, \emph{Equivariant {G}romov-{W}itten
  invariants of algebraic {GKM} manifolds}, SIGMA Symmetry Integrability Geom.
  Methods Appl. \textbf{13} (2017), Paper No. 048, 21. \MR{3667222}

\bibitem[LT93]{LibTei}
A.~Libgober and J.~Teitelbaum, \emph{Lines on {C}alabi-{Y}au complete
  intersections, mirror symmetry, and {P}icard-{F}uchs equations}, Internat.
  Math. Res. Notices (1993), no.~1, 29--39. \MR{1201748}

\bibitem[LV11]{levVa}
D.~Levcovitz and I.~Vainsencher, \emph{Symplectic enumeration}, Bull. Braz.
  Math. Soc. (N.S.) \textbf{42} (2011), no.~3, 347--358. \MR{2833807}

\bibitem[Man95]{manin95}
Yu.~I. Manin, \emph{Generating functions in algebraic geometry and sums over
  trees}, The moduli space of curves ({T}exel {I}sland, 1994), Progr. Math.,
  vol. 129, Birkh\"{a}user Boston, Boston, MA, 1995, pp.~401--417. \MR{1363064}

\bibitem[Mur20]{muratore2020enumeration}
Giosu{\`e} Muratore, \emph{Enumeration of rational contact curves via torus
  actions}, 2020, Michigan Math. J. (to appear).

\bibitem[Mur21]{Muratore_2021}
Giosu\`e Muratore, \emph{A recursive formula for osculating curves}, Ark. Mat.
  \textbf{59} (2021), no.~1, 195--211. \MR{4256011}

\bibitem[{OEIS}]{oeis}
{OEIS Foundation Inc}, \emph{The {O}n-{L}ine {E}ncyclopedia of {I}nteger
  {S}equences}, 2022, http://oeis.org.

\bibitem[Pan13]{Pand}
R.~Pandharipande, \emph{Convex rationally connected varieties}, Proc. Amer.
  Math. Soc. \textbf{141} (2013), no.~5, 1539--1543. \MR{3020841}

\bibitem[Voi96]{voisin1996mathematical}
Claire Voisin, \emph{A mathematical proof of a formula of {A}spinwall and
  {M}orrison}, Compositio Math. \textbf{104} (1996), no.~2, 135--151.
  \MR{1421397}

\bibitem[Wit91]{Witten}
Edward Witten, \emph{Two-dimensional gravity and intersection theory on moduli
  space}, Surveys in differential geometry ({C}ambridge, {MA}, 1990), Lehigh
  Univ., Bethlehem, PA, 1991, pp.~243--310. \MR{1144529}

\bibitem[Ye94]{ye}
Yun-Gang Ye, \emph{A note on complex projective threefolds admitting
  holomorphic contact structures}, Invent. Math. \textbf{115} (1994), no.~2,
  311--314. \MR{1258907}

\end{thebibliography}

\newcommand{\etalchar}[1]{$^{#1}$}
\providecommand{\bysame}{\leavevmode\hbox to3em{\hrulefill}\thinspace}
\providecommand{\MR}{\relax\ifhmode\unskip\space\fi MR }
\providecommand{\MRhref}[2]{%
  \href{http://www.ams.org/mathscinet-getitem?mr=#1}{#2}
}
\providecommand{\href}[2]{#2}

\end{document}